\documentclass[20pt]{amsart}
\usepackage[latin1]{inputenc}
\usepackage{mathrsfs}
\usepackage{cite}
\usepackage{mathtools}
\usepackage{amsmath, amsthm, amsfonts, amssymb}
\usepackage{mathtools}
\usepackage{enumitem}
\usepackage{tikz}
\usepackage{tikz-cd}
\usepackage[all]{xy}

\usepackage{amsthm}
\theoremstyle{plain}
\newtheorem{thm}{Theorem}[section]

\newtheorem{cor}[thm]{Corollary}
\newtheorem{lemma}[thm]{Lemma}
\newtheorem{rmk}[thm]{Remark}

\theoremstyle{definition}
\newtheorem{example}[thm]{Example}
\newtheorem{defn}[thm]{Definition}

 \usepackage{todonotes}

\newcommand{\CHM}{\mathrm{\CHM}}

\usepackage[
urlcolor=blue,
colorlinks=true,
linkcolor=blue,
citecolor=blue,
]{hyperref}

\begin{document}
\title{Bound of automorphisms of fibred surfaces in positive characteristic  \,\, }
\author{Xiaokun Zhong}
	\address{School of Mathematical Sciences, Xiamen University, Xiamen 361005, P. R. China.}
	\email{xiaokunzhong@163.com}

	\makeatletter \@namedef{subjclassname@2020}{\textup{2020} Mathematics Subject Classification} 
     \makeatother
	\thanks{The author was supported in part by NSFC (No.11971399).}
	

\begin{abstract}
{
Let $f \colon S \rightarrow B$ be a fibred surface over an algebraically closed field $k$ of characteristic $p \geq 5$, where $S$ is a minimal smooth projective surface of general type and $B$ is a smooth projective curve of genus $b\geq 2$. We prove that the group of fibration-preserving automorphisms of $f$ has order at most $15658(K_S^2)^{4}$. Furthermore, we provide an example to show that the exponent 4 of the polynomial bound is sharp.
}	
\end{abstract}

\maketitle
\vspace*{6pt}
\tableofcontents  


\section{Introduction}
Let $k$ be an algebraically closed field of characteristic $p \geq 0$. When considering a minimal smooth  projective surface $S$ of general type over $k$, its automorphism group $\mathrm{Aut}_{k}(S)$ becomes a focal point of research. According to Matsumura, the order of $\mathrm{Aut}_{k}(S)$ is finite in \cite{s9}. Bounding the order of $\mathrm{Aut}_{k}(S)$ in terms of $K_S^2$ poses a natural problem.

When $k$ is the complex numbers $\Bbb{C}$, Xiao demonstrated that the order of $\mathrm{Aut}_{k}(S)$ is at most $(42)^2K_{S}^{2}$ for all surfaces $S$ of general type over $\Bbb{C}$ (see \cite{s15,Xiao1995}). Additionally, for a fibration $f\colon S \to B$ over a smooth projective curve of genus $b \geq 2$, and $G:=\{ (\sigma,\varphi) \in \mathrm{Aut}_{k}(S) \times \mathrm{Aut}_{k}(B) | f \circ \sigma =\varphi \circ f \}$ a subgroup of $\mathrm{Aut}_{k}(S)$, Xiao proved that the order of $G$ is at most $882K_S^2$ (refer to \cite{s16}, Proposition 1). 

However, in the case of $k$ having a positive characteristic, the order of $\mathrm{Aut}_{k}(S)$ can not be linearly bounded by $K_S^2$ (see Example 4.9). Finding polynomial bounds (in terms of $K_{S}^{2}$) for the order of $\mathrm{Aut}_{k}(S)$ is a challenging problem. This difficulty persists even for the subgroup $G$ of $\mathrm{Aut}_{k}(S)$.

In the case where $k$ is an algebraically closed field of positive characteristic $p$, the group $\mathrm{Aut}_{k}(S)$ may be not reduced. Fortunately, the order of $\mathrm{Aut}_{k}(S)_{\mathrm{red}}$ (the group scheme $\mathrm{Aut}_{k}(S)$ with reduced structure) is identical to that of $\mathrm{Aut}_{k}(S)$. Let 
\begin{equation}
G:=\{ (\sigma,\varphi) \in \mathrm{Aut}_{k}(S)_\mathrm{red} \times \mathrm{Aut}_{k}(B)\,|\,f \circ \sigma =\varphi \circ f \} 
\end{equation}
be the subgroup of $\mathrm{Aut}_{k}(S)_\mathrm{red}$. In this paper, we investigate the bound of the order of $G$. We provide some necessary preliminaries in Section 2, where we recapitulate the automorphisms group of a curve, the automorphisms group of a fibration, and some inequalities about the canonical divisor $K_S$.

The paper consists of two main parts. The first part, Section 3, considers a regular projective curve $C$ of arithmetic genus $g:=p_a(C) \geq 2$ over a non-perfect field $K$. Assuming that $C$ is geometrically integral and the order of $\mathrm{Aut}_{K}(C)$ is finite. We explore the group structure of $\mathrm{Aut}_{K}(C)$ based on the number of singular points of $C\times_K \overline{K}$, and provide the bound of the order of $\mathrm{Aut}_{K}(C)$ (see Lemma 3.3). The difficult situation arises when $C\times_K \overline{K}$ is a rational curve with a unique singular point. In such a situation, we obtain $\mathrm{Aut}_{K}(C) \cong H \rtimes M$, where $H$ is a Sylow $p$-group and $M$ is a cyclic group with order coprime to $p$. Generally, the $p$-subgroup of $\mathrm{Aut}_{K}(C)$ may be not trivial (see Example 3.9). Estimating the order of $H$ becomes difficult when $p$ divides $2g-2$. While under certain assumptions (see Section 3.2), we show that the order of $H$ has at most $g$.

The second part of this article, Section 4, provides the polynomial bounds for the order of $G$, which is crucial for bounding the subgroup 
\begin{equation}
\mathrm{Aut}_{B}(S):=\{ (\sigma,\varphi) \in G\,|\,\varphi =\mathrm{id}_B \}
\end{equation}
of $G$. Roughly speaking, if we replace $C$ in the first part by the generic fiber of $f$, since $S$ is a minimal smooth surface of general type, then $\mathrm{Aut}_{B}(S) \cong \mathrm{Aut}_{K}(C)$, where $K$ is the function field of $B$. Hence we may regard $\mathrm{Aut}_{K}(C)$ as a subgroup of $G$. In other words, we may regard both $H$ and $M$ as subgroups of $G$ (hence of  $\mathrm{Aut}_{k}(S)_{\mathrm{red}}$) in the difficult situation presented in the first part.

In \cite{s2}, Ballico provided polynomial bounds (in terms of $K_S^2$) for the order of the subgroup of $\mathrm{Aut}_{k}(S)_{\mathrm{red}}$ with order coprime to $p$, with $\frac{45}{2}$ as the exponent. Moreover, Cai showed that every abelian subgroup of $\mathrm{Aut}_{k}(S)_{\mathrm{red}}$ with order coprime to $p$ has order at most $624K_{S}^{2}+6708$ in \cite{s4}. Ballico also investigated the Sylow $p$-subgroups of $\mathrm{Aut}_{k}(S)_{\mathrm{red}}$ in \cite{s3}. 

We briefly revisit Ballico's method in \cite{s3}. Let $F$ be a general fiber of $f$. He considered the line bundle $K_S |_F$ on $F$ (see \cite{s3}, the second Remark 1.5). Since $S$ is a minimal smooth surface of general type,  $\alpha K_S|_F$ is an ample line bundle on $F$ for a sufficiently large integer $\alpha$. Returning to the difficult situation presented in the first part, he claimed that $F$ has no automorphism of order $p$ fixing the isomorphism class of $\alpha K_S |_F$ (refer to \cite{s3}, Remark 1.3). Note that $H$ fixes the isomorphism class of $\alpha K_S |_F$, and $F$ is a general fiber, thus $H$ is trivial. However, his method may not be suitable for the situation where $p$ divides $2p_a(F)-2$ (see Example 3.9).

We summarize the main result of the paper as follows:

\begin{thm}[see Corollary 4.8 and Example 4.9]
Let $f\colon S \to B$ be a fibration over an algebraically closed field $k$ of characteristic $p \geq 5$, where $S$ is a minimal smooth projective surface of general type and $B$ is a smooth projective curve of genus $b\geq 2$. Let $G:=\{ (\sigma,\varphi) \in \mathrm{Aut}_{k}(S)_\mathrm{red} \times \mathrm{Aut}_{k}(B)\,|\,f \circ \sigma =\varphi \circ f \}$. Then the order of $G$ has at most $15658(K_{S}^{2})^{4}$, and the exponent $4$ of the polynomial bound $($in terms of $K_{S}^{2}$$)$ is sharp.
\end{thm}


\section{Preliminaries}

In this section, let $k$ be an algebraically closed field of characteristic $p>0$. Let $S$ be a minimal smooth projective surface of general type, and $B$ a smooth projective curve of genus $b\geq 2$ over $k$.

\subsection{Automorphisms group of a curve}

Let $\mathrm{GL}_{2}(k)$ be the group of invertible $2 \times 2$ matrices with entries in $k$. We write $\mathrm{PGL}_{2}(k):=\mathrm{GL}_{2}(k)/k^{*}$. It is known that $\mathrm{PGL}_{2}(k) \cong \mathrm{Aut}_{k}(\Bbb P^1_{k})$. For a given finite group $H$, we use $|H|$ to denote its order.

\begin{thm}[see \cite{s5}, Lemma 3.2]
Let $s \in \mathrm{PGL}_{2}(k)$ be nontrivial and of finite order. Then $s$ has a unique fixed point on $\Bbb P^1_{k}$ if and only if $s^{p}=\mathrm{id}$.
\end{thm}

\begin{thm}[see \cite{s5}, Proposition 4.7]
Let $G$ be a finite subgroup of $\mathrm{PGL}_2(k)$. Then $G$ fixes a unique point on $\Bbb P^1_{k}$ if and only if $G$ is isomorphic to the semidirect product $H \rtimes M$, where $H$ is a nontrivial Sylow $p$-subgroup of $G$, $M$ is a cyclic subgroup of $G$, and $(p,\,|M|)=1$.
\end{thm}

\begin{cor}
Let $G$ be a finite subgroup of $\mathrm{PGL}_2(k)$. If $G$ fixes at least two points on $\Bbb P^1_{k}$, then $G$ is a cyclic group and $(p,\,|G|)=1$.
\end{cor}
\begin{proof}
If $G$ fixes at least two points on $\Bbb P^1_{k}$, then $s$ fixes at least two points on $\Bbb P^1_{k}$ for any $s \in G$. By Theorem 2.1, the order of $s$ is coprime to $p$. Therefore, $(p,\,|G|)=1$. Furthermore, $G$ is a cyclic group (see \cite{s5}, Proposition 4.5).  
\end{proof}

\begin{thm}[see \cite{s13}, Hauptsatz]
If $C$ is a smooth projective curve of genus $g \geq 2$ over $k$, then $|\mathrm{Aut}_{k}(C)|<16 \cdot g^{4}$, except in the case where the function field $K(C)$ is $k(x,y)$ with relation
\begin{equation}
y^{p^n}+y=x^{{p^{n}+1}}, \ \ p^{n} \geq 3.
\end{equation}
In this exceptional case, one has $g=\frac{1}{2}\cdot p^{n}(p^n-1)$ and 
\begin{equation}
|\mathrm{Aut}_{k}(C)|=p^{3n}(p^{3n}+1)(p^{2n}-1).
\end{equation}
\end{thm}

\begin{cor}
If $C$ is a smooth projective curve of genus $g \geq 2$ over $k$, then $|\mathrm{Aut}_{k}(C)|<75 \cdot g^{4}$.
\end{cor}
\begin{proof} 
In the exceptional case of Theorem 2.4, we have $g=\frac{1}{2}\cdot p^{n}(p^n-1)$ and $|\mathrm{Aut}_{k}(C)|=p^{3n}(p^{3n}+1)(p^{2n}-1)$. Thus, $|\mathrm{Aut}_{k}(C)|=\frac{16 p^{3n}(p^{3n}+1)(p^{2n}-1)}{p^{4n}(p^n-1)^4}g^4=\frac{16 (p^{3n}+1)(p^{n}+1)}{p^{n}(p^n-1)^3}g^4$. Note that the function $h(x):=\frac{(x^3+1)(x+1)}{x(x-1)^3}$ is decreasing in the interval $[3, \infty)$, which implies that $|\mathrm{Aut}_{k}(C)|$ $ \leq {{224} \over {3}} \cdot g^{4}$, and equality holds if and only if $p^{n}=3$. Hence $|\mathrm{Aut}_{k}(C)| \leq {{224} \over {3}} \cdot g^{4}<75 \cdot g^{4}$.
\end{proof}

\begin{thm}[see \cite{s4}, Theorem 2.1]
If $C \subseteq S$ is an integral curve of arithmetic genus $g \geq 2$, then every finite abelian automorphism group $G \subseteq \mathrm{Aut}_{k}(C)$ with order prime to $p$ has order $|G| \leq 8g+4$.
\end{thm}

\subsection{Automorphisms group of a fibration}

We say $f\colon S \to B$ is a fibration if $f$ is a flat morphism and $f_{*}\mathcal{O}_{S}=\mathcal{O}_{B}$. It is known that the generic fiber of $f$ is geometrically integral. In fact, the generic fiber of $f$ is geometrically integral if and only if $K(B)$ is algebraically closed in $K(S)$ (see \cite{s1}, Lemma 7.2, Corollary 7.3 and \cite{s6}, Proposition 5.51). 

Suppose $K(B)$ is not  algebraically closed in $K(S)$. We consider the \textit{Stein factorization} $S\xrightarrow{h} B'\xrightarrow{\delta}B$ of $f$. Thus, $K(B')$ is algebraically closed in $K(S)$. Hence $\delta$ is not birational. This implies that $\mathrm{rank}\,\delta_{*}\mathcal{O}_{B'}>1$, which contradicts the fact that $\delta_{*}\mathcal{O}_{B'}=\delta_{*}h_*\mathcal{O}_S=\mathcal{O}_B$. Therefore, $K(B)$ is algebraically closed in $K(S)$.

Let $F$ be a general fiber of $f$ and $g:=p_a(F)$. 

Denoting $G:=\{ (\sigma,\varphi) \in \mathrm{Aut}_{k}(S)_{\mathrm{red}} \times \mathrm{Aut}_{k}(B) | f \circ \sigma =\varphi \circ f \}$
\begin{displaymath}
\xymatrix{
S\ar[r]^{\sigma} \ar[d]_{f} & S \ar[d]^{f}\\
B\ar[r]_{\varphi} & B 
}
\end{displaymath}
and $\mathrm{Aut}_{B}(S):=\{ (\sigma,\varphi) \in G$ $|$ $\varphi =id_{B} \}$.

We have the following exact sequences
\begin{equation}
1 \rightarrow \mathrm{Aut}_{B}(S) \rightarrow G \xrightarrow{u} \mathrm{Aut}_{k}(B),
\end{equation}
\begin{equation}
1 \rightarrow \mathrm{Aut}_{B}(S) \xrightarrow{v} \mathrm{Aut}_{k}(F),
\end{equation}
where $u(\sigma,\varphi):=\varphi$ and $v(\sigma,id):=\sigma|_{F}$.

According to (5), we have $|G| \leq |\mathrm{Aut}_B(S)| \cdot |\mathrm{Aut}_k(B)|$. Consequently, to bound the order of $G$, it suffices (by Corollary 2.5) to bound the order of $\mathrm{Aut}_B(S)$.

\subsection{Some inequalities about canonical divisor}

Let $f\colon S \to B$ be a fibration, and $K_S$ the canonical divisor of $S$. 

Let $K_{S/B}:=K_{S}-f^{*}K_B$ and $l(f):=\mathrm{dim}_{k}(R^{1}f_{*}\mathcal{O}_{S})_{\mathrm{tor}}$. By Leray spectral sequence, we have
\begin{equation}
\mathrm{deg}(f_{*}\omega _{S/B})= \chi(\mathcal{O}_S)-(g-1)(b-1)+l(f)\geq \chi(\mathcal{O}_S)-(g-1)(b-1).
\end{equation}
In \cite{s7}, Yi Gu, Xiaotao Sun and Mingshuo Zhou have given
\begin{equation}
\chi (\mathcal{O}_S) \geq \frac{1}{32} K_S^2>0.
\end{equation}

The slope of a vector bundle $E'$ is defined as $\mu(E'):=\mathrm{deg}(E')/\mathrm{rank}(E')$. Let $E:=f_{*} \omega _{S/B}$. By choosing a sufficiently large integer $k_{1}$, there is a $Harder$-$Narasimhan$ $filtration$ (see \cite{s7}, page 617)
\begin{equation}
0=E_0\subset E_1\subset \cdots \subset E_n:=F_{B}^{k_{1}*} E
\end{equation} 
such that $\mu_{1}> \cdots > \mu _{n} > \mu _{n+1}:=0$, where $F_{B}$ is the absolute Frobenius morphism of $B$, and $\mu_{i}:=\mu(E_{i}/E_{i-1})$ for $1\leq i \leq n$. Let $r_{i}:=\mathrm{rank}(E_{i})$. Then

\begin{equation}
\begin{aligned}
\mathrm{deg}E&=\frac{1}{p^{k_1}} deg(F_{B}^{k_{1}*} E)\\
&=\frac{1}{p^{k_1}} [r_{1}(\mu _{1}-\mu _{2})+r_{2}(\mu _{2}-\mu _{3})+\cdots+r_{n-1}(\mu _{n-1}-\mu _{n})+r_{n}\mu _{n}]\\
&\leq \frac{1}{p^{k_1}} [r_{n}(\mu _{1}-\mu _{2})+r_{n}(\mu _{2}-\mu _{3})+\cdots+r_{n}(\mu _{n-1}-\mu _{n})+r_{n}\mu _{n}]\\
&=\frac{1}{p^{k_1}} r_{n}\mu _{1}=\frac{g \mu _{1}}{p^{k_1}}
\end{aligned}
\end{equation} 
and there is an effective divisor $Z_1$ on $S$ such that $N_{1}:=p^{k_1}K_{S/B} - Z_{1} - \mu _{1}F$ is nef (see \cite{s14}), where $F$ is a general fiber of $f$.

\begin{lemma}
Let $f\colon S \to B$ be a fibration over an algebraically closed field $k$ of characteristic $p>0$, where $S$ is a minimal smooth projective surface of general type and $B$ is a smooth projective curve of genus $b\geq 2$. Let $F$ be a general fiber of $f$ and $g:=p_{a}(F)$. Then $K_S^2\geq 2(g-1)(b-1)(1+\frac{1}{g})(1+\frac{g-1}{15g+1})$. Moreover, $K_S^2\geq 4(g-1)(b-1)(1+\frac{1}{g})(1+\frac{g-1}{7g+1})$ if the general fiber $F$ is smooth.
\end{lemma} 

\begin{proof} 
Let $E:=f_{*} \omega _{S/B}$.
By (7), (8) and (10), we have 
\begin{equation}
\frac{g \mu _1}{p^{k_1}}  \geq \mathrm{deg}E \geq \chi(\mathcal{O}_S)-(g-1)(b-1) \geq \frac{1}{32}K_S^2 -(g-1)(b-1).
\end{equation} 
Note that $Z_1$ is effective, $N_{1}=p^{k_1}K_{S/B} - Z_{1} - \mu _{1}F$ and $K_{S}$ are both nef. So
\begin{equation}
\begin{aligned}
p^{2k_{1}}K_S^2&=p^{k_1}K_{S}.(p^{k_1}K_{S}-N_{1}+N_{1})\\
&\geq p^{k_1}K_{S}.(p^{k_1}K_{S}-N_{1})\\
&= p^{k_1}K_{S}.[\mu _{1}+2p^{k_1}(b-1)]F+p^{k_1}K_{S}.Z_{1}\\
&\geq p^{k_1}[\mu _{1}+2p^{k_1}(b-1)]K_{S}.F\\
&=p^{k_1}[\mu _{1}+2p^{k_1}(b-1)](2g-2).
\end{aligned}
\end{equation}
Combining (11) with (12), we have
\begin{equation}
\begin{aligned}
K_{S}^{2}&\geq 4(b-1)(g-1) + \frac{2(g-1)\mu_1}{p^{k_1}}\\
&\geq 4(b-1)(g-1)+ \frac{2(g-1)}{g} \left[ \frac{1}{32}K_S^2 -(g-1)(b-1) \right].
\end{aligned}
\end{equation}
Therefore, $K_S^2\geq \frac{2(g-1)(b-1)(2-\frac{g-1}{g})}{1-\frac{g-1}{16g}}=2(g-1)(b-1)(1+\frac{1}{g})(1+\frac{g-1}{15g+1})$.

If the general fiber $F$ is smooth, then $K_{S/B}^{2} \geq {{4g-4}\over{g}} \mathrm{deg} E$ (see \cite{s7}, Theorem 1.1). Since $\mathrm{deg}E \geq \frac{1}{32}K_S^2 -(g-1)(b-1)$, we obtain
\begin{equation}
\begin{aligned}
K_{S}^{2}&=K_{S/B}^{2}+8(g-1)(b-1) \\
&\geq 8(g-1)(b-1)+ \frac{4(g-1)}{g} \left[ \frac{1}{32}K_S^2 -(g-1)(b-1) \right].\\
\end{aligned}
\end{equation} 
Therefore, $K_S^2\geq \frac{4(g-1)(b-1)(2-\frac{g-1}{g})}{1-\frac{g-1}{8g}}=4(g-1)(b-1)(1+\frac{1}{g})(1+\frac{g-1}{7g+1})$.
\end{proof}


\section{Bounding the regular curves over a non-perfect field}
\subsection{Bound of the order of automorphisms group of curves}

For the fibration $f\colon S \to B$ in Section 2.2, we have known that the key to bound the order of $G$ lies in bounding the order of $\mathrm{Aut}_B(S)$. Since $S$ is a minimal smooth projective surface of general type and $g(B) \geq 2$, we obtain $\mathrm{Aut}_B(S) \cong \mathrm{Aut}_{K(B)}(S_{\eta})$, where $K(B)$ is the function field of $B$, $\eta$  is the generic point of $B$, and $S_{\eta}$ is the generic fiber of $f$. Note that $K(B)$ is not algebraically closed, and $S_{\eta}$ is a regular projective geometrically integral curve over $K(B)$ (see Section 2.2).

Based on this, and unless otherwise specified, the discussion in this chapter will proceed under the following assumptions: Let $K$ be an infinite field of characteristic $p>0$, and $C$ a regular projective geometrically integral curve over $K$.

\begin{defn}[see \cite{s10}, Part \uppercase \expandafter{\romannumeral 1}, 2.4]
Let $P$ be a point of $C$, and $X:=C-\{P\}$. For any field extension $K'$ of $K$, the normalization $\widetilde{C}_{K'}$ of $C_{K'}:=C \times_{K}K'$ contains naturally $X \times_{K}K'$. If $\widetilde{C}_{K'}-X \times_{K}K'$ is a single point, then $P$ is called a $geometric$ $one$-$place$ $point$. 
\end{defn}

\begin{lemma}[see \cite{s10}, Part \uppercase \expandafter{\romannumeral 1}, 2.4.2]
There exists a finite separable extension $K'$ of $K$ such that $C':=C \times_{K}K'$ is a regular curve over $K'$ and every non-smooth point of $C'$ is a $geometric$ $one$-$place$ $point$.
\end{lemma}

\begin{lemma}
Assume that $C$ is geometrically integral and $g:=p_{a}(C) \geq 2 $. Let $\overline{C}:=C \times_{K}\overline{K}$. Let $t$ be the number of singular points on $\overline{C}$, and $\pi \colon \widetilde{C} \to \overline{C}$ the normalization. Then the following assertions hold.
\\(i) If $g(\widetilde{C}) \geq 2$, then $|\mathrm{Aut}_K(C)|<75 \cdot g^{4}$.
\\(ii) If $g(\widetilde{C})=1$, then $|\mathrm{Aut}_K(C)| \leq 24 \cdot (g-1)$.
\\(iii) If $g(\widetilde{C})=0$, $t \geq 2$, and $|\mathrm{Aut}_K(C)|$ is finite, then there is a cyclic subgroup $M$ of $\mathrm{Aut}_K(C)$ such that $[\mathrm{Aut}_K(C):M] \leq g(g-1)$ and $(p,\,|M|)=1$.
\\(iv) If $g(\widetilde{C})=0$, $t=1$, and $|\mathrm{Aut}_K(C)|$ is finite, then $\mathrm{Aut}_K(C) \cong H \rtimes M$, where $H$ is a Sylow $p$-subgroup of $\mathrm{Aut}_K(C)$ (which may be trivial), $M$ is a cyclic subgroup of $\mathrm{Aut}_K(C)$, and $(p,\,|M|)=1$.
\end{lemma} 

\begin{proof}
By Lemma 3.2, we may choose a finite separable field extension $K'$ of $K$ such that every non-smooth point of $C':=C \times_{K}K'$ is a $geometric$ $one$-$place$ $point$. The universal property of fibred product implies that $\mathrm{Aut}_K(C) \subseteq \mathrm{Aut}_K(C') \subseteq \mathrm{Aut}_{\overline K}(\overline C)$.

(i) If $g(\widetilde{C}) \geq 2$, then  $|\mathrm{Aut}_K(C)| \leq |\mathrm{Aut}_{\overline K}(\widetilde C)|<75g(\widetilde{C})^{4} \leq 75g^{4}$ by Corollary 2.5.

(ii) If $g(\widetilde{C})=1$, then $\overline{C}$ is singular because $g \geq 2>g(\widetilde{C})$. For any singular point $P\in \overline{C}$, given that the non-smooth points of $C'$ are $geometric$ $one$-$place$ $points$, we obtain that the set $\pi^{-1}(P)$ consists of a single point. Let $P_1$ be a singular point of $\overline{C}$. Let 
\begin{equation}
H_1:=\{ v \in \mathrm{Aut}_{\overline k}(\overline C) | v(P_1)=P_1 \},
\end{equation}
\begin{equation}
H':=\{ v \in \mathrm{Aut}_{\overline k}(\widetilde C) | v(\pi^{-1}(P_1))=\pi^{-1}(P_1) \}.
\end{equation}
Then $H_1 \subseteq H'$. Note that $\widetilde C$ is an elliptic curve and $\pi^{-1}(P_1)$ is a single point. It is a fact that there are at most 24 automorphisms of $\widetilde C$ fixing $\pi^{-1}(P_1)$ (see \cite{s12}, Chapter 3, Theorem 10.1 and Appendix A, Proposition 1.2(c)). Hence $|H_1| \leq |H'| \leq 24$. 

Observe that $[\mathrm{Aut}_{\overline k}(\overline C): H_1] \leq g-1$. Indeed, we may assume that the set of singular points of $\overline{C}$ is $\{ P_i \,|\, 1\leq i \leq t \}$ with $t \geq 2$. Define 
$$
H_i:=\left\{ v \in \mathrm{Aut}_{\overline K}(\overline C) \,|\, v(P_{1})=P_{i} \right\}.
$$
Then we have $\mathrm{Aut}_{\overline K}(\overline C)=\cup_{1\leq i \leq t} H_i$. Suppose there exists $i>1$ such that the set $H_i$ is non-empty. Let $v_i \in H_i$. For any $v \in H_i$, the composition $v_i^{-1} v$ fixes the point $P_1$. This implies $v_i^{-1} v \in H_1$. Consequently, $v \in v_i H_1$, which shows that $H_i$ is a left coset of $H_1$. Therefore, $[\mathrm{Aut}_{\overline K}(\overline C):H_1] \leq t \leq g-g(\widetilde{C})= g-1$. Combining this with the inequality that $|H_1| \leq 24$, we obtain $|\mathrm{Aut}_K(C)| \leq 24(g-1)$. 

(iii) If $g(\widetilde{C})=0$, then $\overline{C}$ is singular, and it has at most $g$ singular points. When $t \geq 2$, we may choose two singular points $P_{1}$ and $P_{2}$ of $\overline{C}$. Regard $\mathrm{Aut}_K(C)$ as a subgroup of $\mathrm{Aut}_{\overline K}(\overline C)$, and let
\begin{equation}
M:=\{ v \in \mathrm{Aut}_K(C) | v(P_{1})=P_{1}\,\,\mathrm{and}\,\, v(P_{2})=P_{2} \},
\end{equation}
\begin{equation}
M_{2}:=\{ v \in \mathrm{Aut}_{\overline K}(\widetilde C) | v(\pi^{-1}(P_{1}))=\pi^{-1}(P_{1})\,\,\mathrm{and}\,\,  v(\pi^{-1}(P_{2}))=\pi^{-1}(P_{2})\}.
\end{equation}
Note that $\mathrm{Aut}_K(C) \subseteq \mathrm{Aut}_{\overline{K}}(\overline{C}) \subseteq \mathrm{Aut}_{\overline K}(\widetilde C) \cong \mathrm{PGL}_{2}(\overline K)$, so $M \subseteq M_2$. 

Both $\pi^{-1}(P_{1})$ and $\pi^{-1}(P_{2})$ are single points, so the finite group $M$ fixes two distinct points of $\widetilde C$. By Corollary 2.3, $M$ is a cyclic group with $(p,|M|)=1$. On the other hand, for any $v \in \mathrm{Aut}_K(C)$, $E:=\{ v(P_{1}),v(P_{2}) \}$ is contained in the set of singular points of $\overline C$. Given that $\overline C$ has at most $g$ singular points, there are at most $g(g-1)$ possibilities for $E$. Therefore, $[\mathrm{Aut}_K(C):M] \leq g(g-1)$. 

(iv) If $g(\widetilde{C})=0$ and $t=1$, then $\overline C$ has a unique singular point $P$. 

Note that $\mathrm{Aut}_K(C) \subseteq \mathrm{Aut}_{K'}(C') \subseteq \mathrm{Aut}_{\overline K}(\overline C) \subseteq \mathrm{Aut}_{\overline K}(\widetilde C) \cong \mathrm{PGL}_{2}(\overline K)$. We can regard $\mathrm{Aut}_{K}(C)$ as a finite subgroup of $\mathrm{Aut}_{\overline K}(\widetilde C)$. In other words, as a subgroup of $\mathrm{Aut}_{\overline K}(\widetilde C)$, $\mathrm{Aut}_K(C)$ fixes the single point $\pi^{-1}(P)$ of $\widetilde C$. By Theorem 2.2, the proof is complete.  
\end{proof}

\begin{lemma}
Let $K$ be an infinite field of characteristic $p>0$, and let $C$ be a regular projective curve over $K$ with a unique non-smooth point $Q=(x_0,y_0)$. Suppose that the local equation of $C$ at $Q$ is of one of the following forms.
\begin{equation*}(\mathrm{I})\,\,\,
\text{When $x_0 \in K$, the local equation is}\,\,\,\,\,\,\,\,\,\,\,\,\,\,\,\,\,\,\,\,\,\,\,\,\,\,\,\,\,\,\,\,\,\,\,\,\,\,\,\,\,\,\,\,\,\,\,\,\,\,\,\,\,\,\,\,\,\,\,\,\,\,\,\,\,\,\,\,\,\,\,\,\,\,\,\,\,\,\,\,\,\,\,\,\,\,\,\,\,\,\,\,\,\,\,\,\,\,\,\,\,\,\,\,\,\,\,\,\,\,\,\,\,\,\,
\end{equation*}
$$
y^{p}=P(x):=bx^{r}+\sum _{i=0}^{m} b_{l_i}x^{l_ip},
$$
where $(p,r)=1$, $r\geq 2$, $b\neq 0$, $0=l_0< l_1 < \cdots < l_m$, $b_{l_i} \in K-K^p$, and $K^p:=\{a^p\,|\,a\in K\}$. 
\begin{equation*}(\mathrm{II})\,\,\,
\text{When $x_0 \notin K$, the local equation is}\,\,\,\,\,\,\,\,\,\,\,\,\,\,\,\,\,\,\,\,\,\,\,\,\,\,\,\,\,\,\,\,\,\,\,\,\,\,\,\,\,\,\,\,\,\,\,\,\,\,\,\,\,\,\,\,\,\,\,\,\,\,\,\,\,\,\,\,\,\,\,\,\,\,\,\,\,\,\,\,\,\,\,\,\,\,\,\,\,\,\,\,\,\,\,\,\,\,\,\,\,\,\,\,\,\,\,\,\,\,\,\,\,\,\,
\end{equation*}
$$
y^{p}=P(x):=bx(x^{p^s}-x_0^{p^s})^t+\varphi(x),
$$
where $s,t\in \Bbb{Z}_{>0}$, $x_0^{p^s}\in K-K^p$, and $\varphi(x) \in K[x^p]$.

Let $H$ be a finite $p$-subgroup of $\mathrm{Aut}_K(C)$. If $p_a(C)\ge 2$ and $|H| > p \cdot (2p_{a}(C)-2)$, then the following assertions hold:

(i) $p$ divides $2p_{a}(C)-2$.

(ii) $C$ has infinitely many $K$-rational points.
\end{lemma} 

\begin{proof} 
(i) Let $\pi \colon C \to C/H$ be the quotient morphism. Since $C$ is a regular projective curve, $C/H$ is also a regular projective curve. By the Hurwitz formula (see \cite{s8}, Theorem 7.4.16), we have
\begin{equation}
2p_{a}(C)-2=\mathrm{deg}(\pi) \cdot (2 p_{a}(C/H)-2)+\sum_{x\in C} d_x \cdot [k(x):K],
\end{equation}
where $d_x$ is the different at $x$, $k(x)$ is the residue field of $x$, and $x$ ranges over the closed points of $C$. Since $p_a(C)\ge 2$ and $\mathrm{deg}(\pi)=|H|> p(2p_a(C)-2)$, by (19), we have $p_a(C/H)=0$. 

If $p=2$, then $p\mid (2p_a(C)-2)$ automatically. Assume that $p\neq 2$. As $p_a(C/H)=0$, we have $C/H \cong \Bbb{P}_K^1$. Thus, for any $y \in C/H$, we obtain
\begin{equation}
\mathrm{deg}(\pi)=\displaystyle{\sum_{x \in \pi^{-1}(y)}}e_{x} \cdot [k(x):k(y)]=r_x  e_{x}  [k(x):k(y)]=r_x   e_{x}  [k(x):K],
\end{equation}
where $k(x)$ (resp. $k(y)$) is the residue field of $x$ (resp. $y$), $e_{x}$ is the ramification index of $x$, and $r_x$ is the number of points of $\pi^{-1}(y)$.

As $Q$ is the unique non-smooth point of $C$, $H$ fixes $Q$. This implies that $r_{x_0}=1$. We now compute $k(x_0)$. It suffices to consider the affine neighborhood $\mathrm{Spec}K[x,y]/(y^p-P(x))$. 

If the local equation of $C$ is of type $(\mathrm{I})$,  then $Q=(0,P(0)^{1/p})$, and the maximal ideal at $Q$ is $(x,y^p-P(0))=(x)$. If the local equation is of type $(\mathrm{II})$, then the maximal ideal at $Q=(x_0, P(x_0)^{1/p})$ is $(x^{p^s}-x_0^{p^s})$. In both cases, $k(x_0)$ is a purely inseparable extension of $K$, and hence $Q$ is a ramification point of $\pi$. Thus, $d_{x_0} \geq e_{x_0}$.

If $\pi$ has no other ramification points, then by (19), we obtain $p \,|\,(2p_{a}(C)-2)$. If $x \in C$ is another ramification point of $\pi$, then $p \,|\,[k(x):K]$. Otherwise, by the assumption that $H$ is a $p$-group and (20): $[k(x):K]=1$. Therefore, $p\,|\,e_x$, and $k(x)$ is automatically a separable extension of $K$. Consequently, $d_x \geq 2e_x-2$ (see \cite{Ser1979}, Chapter 4, Proposition 4). Applying (19) again, we have
\begin{equation}
\begin{aligned}
2p_{a}(C)-2&=-2\cdot \mathrm{deg}(\pi)+\sum_{x\in C} d_x \cdot [k(x):K]\\
&\geq -2\cdot \mathrm{deg}(\pi)+d_{x_0}\cdot [k(x_0):K]+r_xd_x\\
&\geq -2\cdot \mathrm{deg}(\pi)+\mathrm{deg}(\pi)+\frac{2e_x-2}{e_x}\mathrm{deg}(\pi)\\
&\geq -\mathrm{deg}(\pi)+\left(2-\frac{2}{3}\right)\mathrm{deg}(\pi)\\
&\geq \frac{1}{3} \mathrm{deg}(\pi) > 2p_{a}(C)-2,
\end{aligned}
\end{equation}
which is a contradiction! Thus, $p \,|\,(2p_{a}(C)-2)$.

(ii) Assume that $C$ has only finitely many $K$-rational points. Then, we may choose $c \in K$ such that $P(c) \notin K^p$. Consequently, the maximal idea at $x_c:=(c,P(c)^{\frac{1}{p}})$ of $C$ is $(x-c,y^p-P(c))=(x-c)$. This implies that the residue field $k(x_c)$ of $x_c$ is $K[y]/(y^p-P(c))$. Thus, $k(x_c)$ is a purely inseparable extension of $K$ of degree $p$. Therefore, $x_c$ is a ramification point of $\pi$,and $d_{x_c} \geq e_{x_c}$. Since there are infinitely many such $c$, $\pi$ would have infinitely many ramification points, which is a contradiction!
\end{proof}

\subsection{Sylow $p\,$-group}
Based on certain assumptions, this section will bound the order of the $p$-group $H$ in Lemma 3.3 (iv). Throughout this section, assume that $K$ is an infinite field of characteirstic $p \geq 5$, and $C$ is a regular projective curve over $K$. Let $L$ be the function field of $C$. Suppose that $C$ satisfies the following properties:

1) $L/K$ is separably generated, and $K^{\frac{1}{p}}L$ is a rational function field over $K^{\frac{1}{p}}$.

2) $C$ is geometrically integral, and $C \times_{K}\overline K$ is a rational curve with a unique singular point.

3) $\mathrm{Aut}_{K}(C)$ has finite order, and $g :=p_{a}(C) \geq 2$.

4) $C$ has infinitely many $K$-rational points, and $p\,|\,(2g-2)$.

Under the above assumptions, we show that
\begin{thm}
For any nontrivial $p$-subgroup $H$ of $\mathrm{Aut}_{K}(C)$, we have $|H|<g$.
\end{thm}

By Property 1), there exists an element $x$ in $L$ such that $KL^{p} \cong K \langle x \rangle$, where $K \langle x \rangle$ denotes the field extension of $K$ generated by $x$. Thus, $L$ is a purely inseparable extension of degree $p$ over $K\langle x \rangle$. Additionally, since $L/K$ s separably generated, there exists another element $y$ in $L$ such that is a finite separable extension over $K\langle y \rangle$. Consequently, $L$ is both separable and purely inseparable over $K\langle x,y \rangle$, which impies $L=K\langle x,y \rangle$.

Note that $y^{p} \in K\langle x \rangle$, so there exist polynomials $F_{1}(x)$ and $F_{2}(x)$ in $K[x]$ such that $y^p=\frac{F_1(x)}{F_2(x)}$. Then we have
$$
(yF_{2}(x))^{p}=F_{1}(x)F_{2}(x)^{p-1} \in K[x],
$$ 
and $K\langle x,y \rangle=K\langle x,yF_ {2}(x) \rangle$. Replacing $y$ by $yF_{2}(x)$, we may assume that $L=K\langle x,y \rangle$ with
$$
y^{p}=P(x):= \sum _{i=0} ^{n} a_{i}x^{i} \in K[x].
$$

\subsubsection{The local equation of $C$}
$\\$

Through the previous calculations, we obtain that the function field of $C$ is $L=K\langle x,y \rangle$ with $$
y^{p}=P(x):= \sum _{i=0} ^{n} a_{i}x^{i} \in K[x].
$$
Thus,
$$
P'(x)=na_{n}x^{n-1}+ \cdots +2a_{2}x+a_{1}.
$$
By Property 2), $C \times_{K}\overline K$ has a unique singular point. This forces $P'(x)$ to have a unique root $x_0$ in $\overline{K}$. 

Observe that $x_0 \in K$. Otherwise, since
$P'(x)=n_0 a_{n_0}(x-x_0)^{n_0-1}$ with $0<n_0\leq n$ and $P'(x)\in K[x]$, there exist positive integers $s$ and $t$ such that $n_0-1=p^{s}t$ and $(s,t)=1$. Hence, $P'(x)= a_{n_0}(x^{p^s}-x_0^{p^s})^{t}$. Therefore, $x_0^{p^s} \in K$, and $P(x)=x(x^{p^s}-x_0^{p^s})^{t}+\varphi(x)$ for some $\varphi(x)\in K[x^p]$. In this case, the local equation of $C$ is
$$
y^p=x(x^{p^s}-x_0^{p^s})^{t}+\varphi(x).
$$
This yields $g=\frac{1}{2}(p-1)p^s t$, which contradicts the assumption that $p\mid(2g-2)$.

Therefore, we may assume that $C$ has a non-smooth point at $x=0$. Thus,
$$
P'(x)=na_{n}x^{n-1}+ \cdots +2a_{2}x+a_{1}.
$$
Under this assumption, we have $a_{1}=P'(0)=0$. Let $r$ be the smallest integer such that $ra_{r} \neq 0$ with $1<r \leq n$. We have
$$
P'(x)=x^{r-1}[na_{n}x^{n-r}+ \cdots +(r+1)a_{r+1}x+ra_{r}].
$$
By Property 2), $ia_{i}=0$ for all $i \neq r$. Therefore, we may assume that $x$ and $y$ satisfy the following relation
$$
y^{p}=P(x):=bx^{r}+\sum _{i=0}^{m} b_{i}x^{ip},
$$
where $(p,r)=1$ and $b\neq 0$. Note that $C$ is non-smooth at $x=0$, so $(0,P(0)^{1/p})$ is not $K$-rational, and hence $b_0=P(0) \notin K^p$.

If there exists an element $b_{j}$ in $K^{p}$, replacing $y$ by $y-b_j^{\frac{1}{p}}x^j$, we may assume that the local equation of $C$ at its unique non-smooth point is
\begin{equation}
y^p=P(x):=bx^{r}+\sum_{i=0}^{m} b_{l_i}x^{l_ip},
\end{equation}
where $(p,r)=1$, $r \geq 2$, $b\neq 0$, $0=l_0<l_1<\cdots<l_m$, and $b_{l_i} \in  K-K^{p}$.

From (22), we have
$$
g :=p_{a}(C)=\frac{1}{2}(p-1)(r-1).
$$
Thus, we obtain that $p\,|\,(2g-2)$ if and only if $p\,|\,(r+1)$.

\subsubsection{Bounds of the order of Sylow $p$-subgroups}
$\\$

Let $H$ be a Sylow $p$-subgroup of $\mathrm{Aut}_{K}(C)$. Let $r= \lambda p-1$, where $\lambda$ is a positive integer. For any nontrivial element $\sigma \in H$, by Theorem 2.1 and Corollary 2.3, we obtain that the order of $\sigma$ is $p$. Therefore, there exist $c \in K$ and $f_i(x) \in K\langle x \rangle$ such that
\begin{equation}
\left \{
\begin{aligned}
\sigma x=\frac{x}{1+cx}, \,\,\,\,\,\,\,\,\,\,
\\ \sigma y=\sum_{i=0}^{p-1}f_{i}(x)y^{i}.
\end{aligned}
\right.
\end{equation}

\begin{lemma}
For any $\sigma \in H$, there exist $c \in K$ and $f_0(x) \in K \langle x \rangle$ such that
\begin{equation}
\left \{
\begin{aligned}
\sigma x=\frac{x}{1+cx}, \,\,\,\,\,\,\,\,\,\,\,\,\,\,\,\,\,\,\,\,\,\,\,\,\,\,\,\,\,
\\ \sigma y=f_0(x)+\frac{y}{(1+cx)^{\lambda}}.
\end{aligned}
\right.
\end{equation}
Moreover, $f_0(x)$ satisfies the following relation:
\begin{equation}
f_0^p+\frac{\sum _{i=0}^{m} b_{l_i}x^{l_ip}}{(1+cx)^{\lambda p}}-b_{l_0}=\frac{bcx^{\lambda p}}{(1+cx)^{\lambda p}}+\sum_{i=1}^{m} b_{l_i}(\frac{x}{1+cx})^{l_ip}.
\end{equation}
\end{lemma}

\begin{proof} 
For any $\sigma \in H$, by the preceding preparations, we know there exist $c \in K$ and $f_i(x) \in K\langle x \rangle$ such that $\sigma$ satisfies (23). Substituting $\sigma x=\frac{x}{1+cx}$ into (22) and noting that $r=\lambda p-1$, we have
\begin{equation}
\begin{aligned}
(\sigma y)^{p}-y^{p}&=b[(\frac{x}{1+cx})^{r}-x^{r}]+\sum_{i=1}^{m} b_{l_i}[(\frac{x}{1+cx})^{l_ip}-x^{l_ip}]
\\&=b[\frac{1}{(1+cx)^{\lambda p}}-1]x^{r}+\frac{bcx^{\lambda p}}{(1+cx)^{\lambda p}}+\sum_{i=1}^{m} b_{l_i}[(\frac{x}{1+cx})^{l_ip}-x^{l_ip}].
\end{aligned}
\end{equation}
Substituting $\sigma y=\sum_{i=0}^{p-1}f_{i}(x)y^{i}$ into (22), we obtain
\begin{equation}
\begin{aligned}
(\sigma y)^{p}-y^{p}=&f_{0}^{p}(x)+(f_{1}^{p}(x)-1)(bx^{r}+\sum _{i=0}^{m} b_{l_i}x^{l_ip})+
\\&f_{2}^{p}(x)(bx^{r}+\sum _{i=0}^{m} b_{l_i}x^{l_ip})^{2}+\cdots +f_{p-1}^{p}(x)(bx^{r}+\sum _{i=0}^{m} b_{l_i}x^{l_ip})^{p-1}.
\end{aligned}
\end{equation}
Combining (26) and (27), we have
\begin{equation}
\begin{aligned}
&\,\,\,\,\,\,\,\,b[\frac{1}{(1+cx)^{\lambda p}}-1]x^{r}+\frac{bcx^{\lambda p}}{(1+cx)^{\lambda p}}+\sum_{i=1}^{m} b_{l_i}[(\frac{x}{1+cx})^{l_ip}-x^{l_ip}]
\\&=f_{0}^{p}(x)+(f_{1}^{p}(x)-1)(bx^{r}+\sum _{i=0}^{m} b_{l_i}x^{l_ip})+f_{2}^{p}(x)(bx^{r}+\sum _{i=0}^{m} b_{l_i}x^{l_ip})^{2}
\\&\,\,\,\,\,\,\,\,+\cdots +f_{p-1}^{p}(x)(bx^{r}+\sum _{i=0}^{m} b_{l_i}x^{l_ip})^{p-1}.
\end{aligned}
\end{equation}

On the right hand side of (28), since $(p,r)=1$, the terms
$$
(f_{1}^p(x)-1)\cdot x^{r},\,\,\,f_{2}^{p}(x)\cdot x^{2r},\,\,\,\cdots,\,\,\,\text{and}\,\,\,f_{p-1}^{p}(x)\cdot x^{(p-1)r}
$$
do not lie in $K \langle x^p \rangle$. Furthermore, note that for $1 \leq i \leq j \leq p-1$ with $i \neq j$, we have $ir \not\equiv jr \,(mod\,\, p)$. By comparing both sides of (28), we obtain 
$$
f_ {2}(x)=\cdots=f_{p-1}(x)=0, \,\,\text{and}\,\,\,f_{1}(x)=\frac{1}{(1+cx)^\lambda}.
$$
Substituting this into (28), we get
$$
f_0^p+[\frac{1}{(1+cx)^{\lambda p}}-1] \sum _{i=0}^{m} b_{l_i}x^{l_ip}
=\frac{bcx^{\lambda p}}{(1+cx)^{\lambda p}}+\sum_{i=1}^{m} b_{l_i}[(\frac{x}{1+cx})^{l_ip}-x^{l_ip}].
$$
Adding $\sum_{i=1}^{m}b_{l_i}x^{l_ip}$ to both sides immediately yields
$$
f_0^p+\frac{\sum _{i=0}^{m} b_{l_i}x^{l_ip}}{(1+cx)^{\lambda p}}-b_{l_0}=\frac{bcx^{\lambda p}}{(1+cx)^{\lambda p}}+\sum_{i=1}^{m} b_{l_i}(\frac{x}{1+cx})^{l_ip}.
$$
\end{proof}

\begin{lemma}
  If $l_m>\lambda$, then $\mathrm{Aut}_{K}(C)$ has no nontrivial $p$-subgroups.
\end{lemma}

\begin{proof}
Suppose that there exists an automorphism $\sigma$ of order $p$ in $\mathrm{Aut}_{K}(C)$. By (25), we have
\begin{equation}
\begin{aligned}
f_0^p (1+cx)^{l_mp}=&(1+cx)^{(l_m -\lambda) p}(bcx^{\lambda p} -\sum _{i=0}^{m} b_{l_i}x^{l_ip})+
\\&\sum_{i=1}^{m} b_{l_i} {x}^{l_ip}(1+cx)^{(l_m-l_i)p}+b_{l_0}(1+cx)^{l_m p}.
\end{aligned}
\end{equation}

On the right hand side of equation (29), the monomial $b_{l_m}c^{(l_m -\lambda)p} \cdot x^{(2l_m -\lambda)p}$ is the highest-degree term. On the left hand side, however, we have
$$
f_0^p (1+cx)^{l_m p}\in K^p[x^p].
$$
By comparing both sides of (29), we obtain
$$
b_{l_m}c^{(l_m -\lambda)p} \in K^p.
$$
Note that $b_{l_m} \notin K^p$, so $c=0$. Combining this with (29), we obtain $f_0=0$. Therefore, $\sigma=\mathrm{id}$, which contradicts the assumption that $\sigma$ is nontirvial.
\end{proof}

\begin{lemma}
If $l_m=\lambda$, then every $p$-subgroup of $\mathrm{Aut}_K(C)$ has order at most $g$.
\end{lemma}

\begin{proof}
From (25), we have
\begin{equation}
f_0^p (1+cx)^{\lambda p}=bcx^{\lambda p} +\sum_{i=1}^{m-1} b_{l_i} {x}^{l_ip}[(1+cx)^{(\lambda-l_i)p}-1]+b_{l_0}[(1+cx)^{\lambda p}-1].
\end{equation}

Step 1. We consider a special case.
$$
\text{Case (1):\,} \lambda p,\,(\lambda -l_1) p,\,(\lambda -l_2) p,\,\cdots,\,\text{and}\,(\lambda -l_{m-1}) p\,\,\text{are all povers of\,\,$p$.}
$$
In this case, we obtain
$$
(1+cx)^{(\lambda-l_i)p}-1=(cx)^{(\lambda-l_i)p}.
$$
By (30), we have
\begin{equation}
\begin{aligned}
f_0^p (1+cx)^{\lambda p}=&bcx^{\lambda p} +b_{l_0}(cx)^{\lambda p} +
x^{\lambda p}\sum_{i=1}^{m-1} b_{l_i} c^{(\lambda -l_i)p}
\\=&(bc+b_{l_0} c^{\lambda p}+ \sum_{i=1}^{m-1} b_{l_i} c^{(\lambda -l_i)p} ) \cdot x^{\lambda p}.
\end{aligned}
\end{equation}

If there exists an automorphism of $C$ satisfying (24), then there exists an element $d \in K$ such that \begin{equation}
\left \{
\begin{aligned}
bc+b_{l_0} c^{\lambda p}+ \sum_{i=1}^{m-1} b_{l_i} c^{(\lambda -l_i)p}=d^p,
\\ f_0(x)=\frac{dx^\lambda}{(1+cx)^\lambda}.
\end{aligned}
\right.
\end{equation}
Conversely, if there exist elements $c,d \in K$ such that
$$
bc+b_{l_0} c^{\lambda p}+ \sum_{i=1}^{m-1} b_{l_i} c^{(\lambda -l_i)p}=d^p,
$$
then there exists an automorphism $\sigma$ of $C$ of order $p$ defined by
\begin{equation}
\left \{
\begin{aligned}
\sigma x=\frac{x}{1+cx},\,\,\,\,\,\,\,\,
\\ \sigma y=\frac{dx^{\lambda}+y}{(1+cx)^{\lambda}}.
\end{aligned}
\right.
\end{equation}

According to property 4), $C$ has  infinitely many $K$-rational points. Taking any two distinct $K$-rational points $(x_0,y_0)$ and $(x_1,y_1)$, we obtain
\begin{equation}
 y_j^p=bx_j^r+b_{l_0}+\sum_{i=1}^{m} b_{l_i}(x_j)^{l_ip},\,\, j\in \{0,1\}.
\end{equation}
Note that $C$ is non-smooth at $x=0$ , so $x_0$ and $x_1$ are both non-zero. Thus, we have
\begin{equation}
  \left(\frac{y_j}{x_j^{\lambda}}\right)^p=\frac{b}{x_j}+
  \frac{b_{l_0}}{x_j^{\lambda p}}+\sum_{i=1}^{m-1} \frac{b_{l_i}}{(x_j)^{\lambda p-l_ip}}+b_{l_m}, \,\,j\in \{0,1\}.
\end{equation}

Since $\lambda p$, $(\lambda -l_1) p$, $(\lambda -l_2) p$, $\cdots$, and $(\lambda -l_{m-1}) p$ are all powers of $p$, it follows that
\begin{equation}
\begin{aligned}
 \left(\frac{y_0}{x_0^{\lambda}}-\frac{y_1}{x_1^{\lambda}} \right)^p=&b\left(\frac{1}{x_0}-\frac{1}{x_1}\right)+
  b_{l_0}\left(\frac{1}{x_0}-\frac{1}{x_1}\right)^{\lambda p} +
\\& \sum_{i=1}^{m-1} b_{l_i} \left(\frac{1}{x_0}-\frac{1}{x_1}\right)^{\lambda p-l_i p}.
\end{aligned}
\end{equation}
Consequently, we obtain an automorphism $\sigma_0$ of $C$ order $p$ defined by
\begin{equation}
\left \{
\begin{aligned}
\sigma_0 x=\frac{x}{1+c_0x},\,\,\,\,\,\,\,
\\ \sigma_0 y=\frac{d_0 x^{\lambda}  +y}{(1+c_0x)^{\lambda}},
\end{aligned}
\right.
\end{equation}
where $c_0:=\frac{1}{x_0}-\frac{1}{x_1}$ and $d_0:=\frac{y_0} {x_0^{\lambda}}-\frac{y_1}{x_1^{\lambda}}$.

However, since $C$ has infinitely many $K$-rational points, the above method allows us to construct infinitely many automorphisms of $C$. This contradicts the assumption that $|\mathrm{Aut}_K(C)|$ is finite. Therefore, we conclude that $\lambda p$, $(\lambda -l_1) p$, $\cdots$, and $(\lambda -l_{m-1}) p$ cannot all be powers of $p$.

Step 2. Returning to (30), we consider the second highest degree monomial on the right hand side. Note that when $(\lambda-l_i)p$ is a power of $p$, the term $b_{l_i}x^{l_i p}[(1+cx)^{\lambda p- l_i p}-1]=b_{l_i} c^{\lambda p- l_i p} \cdot x^{\lambda p}$ is the highest degree monomial. Let
$$
(1+cx)^{\lambda p}-1=c^{\lambda p}x^{\lambda p}+ e_0 c^{\theta_0 p} x^{\theta_0 p}+\text{(lower degree terms)},
$$ 
where $x^{\theta_0 p}$ is the monomial in the expansion of $(1+cx)^{\lambda p}-1$ with the degree immediately below $x^{\lambda p}$ (when $\lambda p$ is a power of $p$, we set $\theta_0 p=0$), and $e_0$ is the combinatorial coefficient corresponding to $(cx)^{\theta_0 p}$ in the expansion. By definition, $e_0$ is an integer and thus $e_0 \in K^p$.

Similarly, for $1\leq i \leq m-1$, let
$$
(1+cx)^{(\lambda -l_i)p}-1=c^{(\lambda -l_i)p}x^{(\lambda -l_i)p}+ e_i c^{\theta_i p} x^{\theta_i p}+\,\text{(lower degree terms)}.
$$
where $x^{\theta_i p}$ is the monomial in the expansion of $(1+cx)^{(\lambda -l_i)p}-1$ with the degree immediately below $x^{(\lambda -l_i) p}$ (when $\lambda p-l_i p$ is a power of $p$, we set $\theta_i p=0$), and $e_i$ is the combinatorial coefficient corresponding to $(cx)^{\theta_i p}$ in the expansion. By definition, $e_i \in K^p$.

Since $\lambda p$, $(\lambda -l_1) p$, $(\lambda -l_2) p$, $\cdots$, and $(\lambda -l_{m-1}) p$ are not all powers of $p$, there must exist some $i$ ($0 \leq i \leq m-1$) such that the term $e_i c^{\theta_i p} x^{\theta_i p}$ is non-constant.

Define 
$$
\theta p= max\,_{0\leq i \leq m-1}\, \{(l_i+\theta_i) p \,|\, \theta_i p \neq 0 \}.
$$
Then $0 < \theta p < \lambda p$. Note that in (22), we explicitly set $l_0=0$. We may express the set $\{ l_i \,\,|\,\, (l_i+ \theta_i)p=\theta p \,\,\text{and}\,\, \theta_i p \neq 0\}$ as $\{l_{s_1}, l_{s_2}, \cdots, l_{s_h}\}$, where $s_1<s_2<\cdots<s_h$.

By definition,
$$
\theta p=(\theta_{s_1}+l_{s_1})p=\cdots=(\theta_{s_h}+l_{s_h})p.
$$
Thus, $\theta_{s_1}>\theta_{s_2}>\cdots>\theta_{s_h}\geq 1$.

We now observe that the coefficient of the monomial $x^{\theta p}$ on the right hand side of (30) is
\begin{equation}
b_{l_{s_1}} e_{l_{s_1}} c^{\theta_{s_1}p}+b_{l_{s_2}} e_{l_{s_2}} c^{\theta_{s_2}p} +\cdots + b_{l_{s_h}} e_{l_{s_h}} c^{\theta_{s_h}p}.
\end{equation}
Furthermore, since the left hand side of (30) is $f_0^p(1+cx)^{\lambda p} \in K^p[x^p]$, it follows that 
\begin{equation}
b_{l_{s_1}} e_{l_{s_1}} c^{\theta_{s_1}p}+b_{l_{s_2}} e_{l_{s_2}} c^{\theta_{s_2}p} +\cdots + b_{l_{s_h}} e_{l_{s_h}} c^{\theta_{s_h}p} \in K^p.
\end{equation}

We assert that there exist at most $\theta_{s_1}$ distinct nonzero elements $c$ in $K$ satisfying (39). Otherwise, suppose there exist elements $c_1, c_2, \cdots, c_{\theta_{s_1}+1}$ in $K$ such that
\begin{equation}
b_{l_{s_1}} e_{l_{s_1}} c_i^{\theta_{s_1}p}+b_{l_{s_2}} e_{l_{s_2}} c_i^{\theta_{s_2}p} +\cdots + b_{l_{s_h}} e_{l_{s_h}} c_i^{\theta_{s_h}p}\in K^p, \,\, (1\leq i \leq \theta_{s_1}+1).
\end{equation}
Then
\begin{equation}
R \, \alpha
= \left(
\begin{array}{cccc}
d_{1} & d_{2} & \cdots & d_{\theta_{s_1}+1} \\
\end{array}
\right)^{T},
\end{equation}
where $d_{i}\in K^p$ ($1 \leq i \leq \theta_{s_1}+1$),
$$
\alpha_i:=\left(
\begin{array}{ccccc}
b_{l_{s_i}}e_{l_{s_i}} & 0 & \cdots & 0 \\
\end{array}
\right), (1 \leq i \leq h-1)
$$
is a row vector of order $\theta_{s_i}-\theta_{s_{i+1}}$,
$$
\alpha_{h}:=\left(
\begin{array}{ccccc}
b_{l_{s_h}}e_{l_{s_h}} & 0 & \cdots & 0 \\
\end{array}
\right)
$$
is a row vector of order $\theta_{s_h} +1$,
$$
\alpha^T:=\left(
\begin{array}{cccc}
\alpha_1 & \alpha_2 & \cdots & \alpha_h\\
\end{array}
\right)
$$
is a row vector of order $\theta_{s_1}+1$ formed by concatenating the $h$ row vectors $\alpha_1$, $\cdots$, $\alpha_h$, and
\begin{equation}
R:=\left(
\begin{array}{ccccc}
c_1^{\theta_{s_1}p}&c_1^{(\theta_{s_1}-1)p}&\cdots&c_1^{p}&1\\
\vdots & \vdots & \ddots & \vdots & \vdots\\
c_{\theta_{s_1}+1}^{\theta_{s_1}p}&c_{\theta_{s_1}+1}^{(\theta_{s_1}-1)p}
&\cdots&c_{\theta_{s_1}+1}^{p}&1
\end{array} 
\right)
\end{equation}
is a square matrix of order $\theta_{s_1}+1$.

Notably, all entries of the square matrix $R$ lie in $K^p$, and
\begin{equation}
0\neq \mathrm{det}(R)=\prod _{1\leq i<j\leq \theta_{s_1}+1}(c_{i}-c_{j})^{p} \in K^{p}, 
\end{equation}
Consequently, the entries of $R^{-1}$ also belong to $K^p$. This implies $b_{l_{s_1}} e_{l_{s_1}}\in K^p$. However, since $0 \neq e_{l_{s_1}}\in K^p$, it follows that $b_{l_{s_1}} \in K^p$, which is a contradiction.

Therefore, for any Sylow $p$-subgroup $H$ of $\mathrm{Aut}_K(C)$, we have $|H| \leq \theta_{s_1}+1$. Furthermore, since $p \geq 5$ and
\begin{equation}
g=\frac{1}{2}(p-1)(r-1)=\frac{1}{2}(p-1)(\lambda p-2)>\theta_{s_1}+1,
\end{equation}
we obtain that $|H| \leq \theta_{s_1}+1 <g$.
\end{proof}

Similar to Lemma 3.8, when $l_m< \lambda$, by considering (25), we have that every $p$-subgroup of $\mathrm{Aut}_{k}(X)$ has order at most $g$.

By combining Lemmas 3.7 and 3.8, we thus complete the proof of Theorem 3.5.

\subsubsection{Nontrivial $p$-subgroups}
There is a fibration $S \to B$ such that its generic fiber is a rational curve with a unique non-smooth point, and $\mathrm{Aut}_B(S)$ has a nontrivial $p$-subgroup. We construct it as follows.

\begin{example}
Let $k$ be an algebraically closed field of characteristic $p \geq 5$, and let $K:=k(u,v)/(u^p=uv^{p-1}+v)$ be the function field of a smooth projective curve $B$ over $k$, where $B$ is defined by the equation
$$
u^p=uv^{p-1}+vw^{p-1}.
$$

Let $C$ be a regular curve in $\Bbb{P}_K^2$ defined by the equation
$$
y^p=x^{p-1}z+\theta z^p.
$$
where $\theta:=\frac{u}{v}$, and which does not lie in $K^{p}$.

By the Jacobian criterion, the point $(x:y:z)= (0:\theta^{\frac{1}{p}}:1)$ is the unique non-smooth point on $C$. Furthermore, it is straightforward to compute that $p_a(C)=\frac{1}{2}(p-1)(p-2)$, and the local equation of $C$ at the point $(0:\theta^{\frac{1}{p}}:1)$ is
$$
y^p=x^{p-1}+\theta.
$$

Now, define
\begin{equation}
\left \{
\begin{aligned}
\sigma x=\frac{x}{1+vx},\\
\sigma y=\frac{y+ux}{1+vx}.
\end{aligned}
\right.
\end{equation}
Then
$$
(\sigma y)^p =\frac{y^p+u^px^p}{1+v^px^p}
=\frac{x^{p-1}+\theta+u^px^p}{1+v^px^p}
=\frac{x^{p-1}+\theta+(uv^{p-1}+v)x^p}{1+v^px^p},
$$
$$
\begin{aligned}
(\sigma x)^{p-1}+\theta&=\frac{x^{p-1}}{(1+vx)^{p-1}}+\theta
\\&=\frac{x^{p-1}+vx^p+\theta+\theta v^px^p}{1+v^px^p}
\\&=\frac{x^{p-1}+\theta+(uv^{p-1}+v)x^p}{1+v^px^p}.
\end{aligned}
$$
Therefore, $(\sigma y)^p=\sigma(x^{p-1}+\theta)$. This implies that $\sigma$ is an automorphism of $C$. Moreover, a direct computation confirms that $\sigma$ has order $p$.

We construct the fibration $S \to B$.  First, we have the following commutative diagram:
\begin{displaymath}
C \hookrightarrow \xymatrix{
\Bbb{P}_{K}^{2}\cong \Bbb{P}_{B}^{2} \times_{B} \mathrm{Spec}\,K \ar[r] \ar[d] & \mathrm{Spec}\,K \ar[d]\\
\Bbb{P}_{B}^{2} \ar[r]^\pi & B 
}
\end{displaymath}
where $\pi$ is the projection from the fibred product $\Bbb{P}_{B}^{2}\cong \Bbb{P}_{k}^{2} \times_{\mathrm{Spec}\,k} B$ to $B$.

Let $S_0$ be the Zariski closure of $C$ in $\Bbb{P}_{B}^{2}$. Then $\pi|_{S_0} \colon S_0 \to B$ is a fibration with $C$ as its generic fiber. Let $\widetilde{\pi} \colon S \to B$ be the minimal resolution of $\pi|_{S_0}$. Since $C$ is regular, it remains the generic fiber of $\widetilde{\pi}$. 

Finally, since $p_a(C)={{1}\over{2}}(p-1)(p-2)>2$ and $g(B) \geq 2$, $S$ is a minimal smooth projective surface of general type. Consequently, $\mathrm{Aut}_B(S) \cong \mathrm{Aut}_K(C)$, which has a nontrivial $p$-subgroup generated by $\sigma$.
\end{example}


\section{Bound of automorphisms group of fibred surfaces}
Let $k$ be an algebraically closed field of characteristic $p>0$, $S$ a minimal smooth projective surface of general type over $k$, and $B$ a smooth projective curve over $k$.
Suppose that $f \colon S \to B$ is a fibration, and $b:=g(B)\geq 2$. 

\begin{thm}[see \cite{s11}, Proposition 3, Corollary 5 and Corollary 6]
Assume that the general fiber of $f$ is a rational curve with a unique singular point. Then there is a positive integer $n$ such that 
\\(i) the following diagrams are commutative;
\begin{displaymath}
\xymatrix{
S_{0} \ar[r]^{h_{0}} \ar[d]_{f_{0}} & S_{1} \ar[r]^{h_{1}} \ar[d]_{f_{1}}& \cdots 
\ar[r]^{h_{n-2}}& S_{n-1} \ar[r]^{h_{n-1}} \ar[d]_{f_{n-1}} & S_{n}=S \ar[d]^{f_{n}=f}\\
B_{0} \ar[r]^{F_{0}} & B_{1}\ar[r]^{F_{1}} & \cdots  \ar[r]^{F_{n-2}} & B_{n-1} \ar[r]^{F_{n-1}} & B_{n}=B 
}
\end{displaymath}
where $F_{i}\colon B_{i} \to B_{i+1}$ is the relative Frobenius morphism, $S_{i}$ is the normalization of $S_ {i+1} \times_{B_{i+1}}B_{i}$, and $h_{i}\colon S_{i} \to S_{i+1}$ be the composition of the normalization $S_{i} \to S_{i+1} \times_{B_{i+1}}B_{i}$ and the projection $S_{i+1} \times_{B_{i+1}}B_{i} \to S_{i+1}$. 
\\(ii) the general fiber of $f_{0}$ is a smooth rational curve;
\\(iii) the general fiber of $f_{1}$ is a rational curve with a unique non-smooth point of type $y^{p}=x^{r}$, where $(p,r)=1$ and $r\geq2$. 
\end{thm}

\begin{rmk}
Let $K_i$ be the function field of $B_i$, and $C_i$ the generic fiber of $f_i$. Then $K_i=K_{i+1}^{\frac{1}{p}}$. For $0 \leq i \leq n-1$, $C_i$ is the normalization of $C_{i+1}$ $\times_{K_{i+1}}K_{i}$, and 
\begin{equation}
\mathrm{Aut}_{K_{i+1}}(C_{i+1}) \subseteq \mathrm{Aut}_{K_{i}}(C_{i}).
\end{equation}
\end{rmk}

\begin{rmk}
For $0 \leq i \leq n$, the genus of $B_{i}$ is $b$. The generic fiber $C_1$ of $f_1$ has arithmetic genus $\frac{1}{2}(p-1)(r-1)$. In particular, $p_a(C_1) \geq 2$ when $p \geq 5$. 
\end{rmk}

\begin{rmk}
When $p\geq 5$, by Theorem 4.1 and Remark 4.3, the generic fiber $C_1$ of $f_1$ satisfies the properties 1), 2) and 3) in Section 3.2.
\end{rmk}

\begin{cor}
Let $\widetilde{f}_{i}\colon \widetilde{S}_{i} \to B_{i}$ be the minimal resolution of $f_{i}\colon S_{i} \to B_{i}$. Then the following assertions hold.

(i) The generic fiber of $\widetilde{f}_{i}$ is the same as that of $f_{i}$. 

(ii) Suppose $p \geq 5$. Then $\widetilde{S}_{i}$ is a minimal smooth projective surface of general type for $1 \leq i \leq n$ and $\widetilde{S}_{0}$ is a geometrically ruled surface. Moreover, we have $\mathrm{Aut}_{B_{i}}(\widetilde{S}_{i}) \cong \mathrm{Aut}_{K_{i}}(C_{i})$ for $1 \leq i \leq n$.
\end{cor}
\begin{proof} (i) It is obvious.

(ii) For $1 \leq i \leq n$, since there is no $(-1)$-rational curve contracted by $\widetilde{f}_{i}$ and $g(B_{i})=b \geq 2$, we have $\widetilde{S}_{i}$ is a minimal projective surface. Let $C_{i}$ be the generic fiber of $\widetilde{f}_{i}$. Given that $p \geq 5$, we have
\begin{equation}
p_{a}(C_{i}) \geq p_{a}(C_{1})={{1} \over {2}}(p-1)(r-1) \geq 2, 
\end{equation}
hence $\widetilde{S}_{i}$ is a surface of general type. 

As for $\widetilde{S}_{0}$, since the generic fiber of $\widetilde{f}_{0}$ is a smooth rational curve, we obtain that $\widetilde{S}_{0}$ is a geometrically ruled surface. For $1 \leq i \leq n$, since $\widetilde{S}_{i}$ is the unique minimal model and $b=g(B_i) \geq 2$, we have $\mathrm{Aut} _{B_{i}}(\widetilde{S}_{i}) \cong \mathrm{Aut}_{K_{i}}(C_{i})$.
\end{proof}

\begin{lemma}
Let $C_{1}$ be the generic fiber of $\widetilde{f}_{1}: \widetilde{S}_{1} \to B_{1}$. Assume that $\overline{C}_{1}:=C_{1} \times_{K_{1}}\overline{K}_{1}$ has a unique singular point and $p \geq 5$. Then the following assertions hold.
\\(i) $\mathrm{Aut}_{B_{1}}(\widetilde{S}_{1}) \cong \mathrm{Aut}_{K_{1}}(C_{1}) \cong H \rtimes M$, where $H$ is a Sylow $p$-subgroup of $\mathrm{Aut}_{K_{1}}(C_{1})$ $($which may be trivial $)$, $M$ is a cyclic subgroup of $\mathrm{Aut}_{K_{1}}(C_{1})$ and $(p,\,|M|)=1$.
\\(ii) $|H|\leq 2(g-1)(2g+1)$, where $g:=p_{a}(C_1)$.
\\(iii) $|\mathrm{Aut}_{B_{1}}(\widetilde{S}_{1})| \leq 8(g-1)(2g+1)^2$
\end{lemma}

\begin{proof}
(i) This follows directly from Corollary 4.5 and Lemma 3.3 (iv).

(ii) Note that $\mathrm{Aut}_{K_{1}}(C_{1}) \subseteq \mathrm{Aut}_{\overline{K}_{1}}(\overline{C}_{1}) \subseteq \mathrm{PGL}_2(\overline{K}_{1})$. Since $\overline{C}_1$ is a rational curve over $\overline{K}_1$ with a unique singular point, Tate's formula for genus change (see \cite{Tate1952}) implies that $(p-1)\,|\,2g$. Consequently, $p \leq 2g+1$.  Moreover, $C_1$ satisfies properties 1), 2) and 3) in Section 3.2. Combining Lemma 3.4 and Theorem 3.5, we obtain
$$
|H| \leq p(2g-2) \leq 2(g-1)(2g+1).
$$ 

(iii) By (i), we have $\mathrm{Aut}_{B_{1}}(\widetilde{S}_{1}) \cong H \rtimes M$, where $H$ is a $p$-group and $M$ is a cyclic group of order coprime to $p$. To bound the order of $\mathrm{Aut}_{B_{1}}(\widetilde{S}_{1})$, it suffices to bound the order of $M$.

Let $F$ be a general fiber of $\widetilde{f}_1$. We have the following exact sequence 
$$
1 \rightarrow \mathrm{Aut}_{B_1}(\widetilde{S}_1) \rightarrow \mathrm{Aut}_{k}(F). 
$$
By Theorem 2.6, $|M| \leq 8g+4$. Combining this with (ii), we obtain
$$
|\mathrm{Aut}_{B_1}(\widetilde{S}_1)| \leq 8(g-1)(2g+1)^2.
$$
\end{proof}

\begin{thm}
Let $f\colon S \rightarrow B$ be a fibration over an algebraically closed field k of characteristic $p > 0$, where S is a minimal smooth projective surface of general type and B is a smooth projective curve of genus $b\geq 2$. Let $C$ be the generic fiber of $f$, and $g:=p_{a}(C)$. Let $K$ be the function field of $B$, and $\overline{C}:=C \times_{K}\overline{K}$. Let $t$ be the number of singular points on $\overline{C}$, and $\pi \colon \widetilde{C} \to \overline{C}$ the normalization. We have
\\(i) if $g(\widetilde{C}) \geq 2$, then $|G| < 5625 \cdot g^{4} \cdot b^{4}$.
\\(ii) if $g(\widetilde{C})=1$, then $|G|<1800 \cdot (g-1) \cdot b^{4}$.
\\(iii) if $g(\widetilde{C})=0$, and $t \geq 2$, then $|G|<300 \cdot g \cdot (g-1) \cdot (2g+1) \cdot b^{4} \leq 3000 \cdot (g-1)^3 \cdot b^4$.
\\(iv) if $g(\widetilde{C})=0$, $t=1$, and $p\geq5$, then $|G|<600 \cdot (g-1)(2g+1)^2 \cdot b^4$. 
\end{thm}

\begin{proof} Note that $\mathrm{Aut}_{K}(C) \cong \mathrm{Aut}_{B}(S)$ has finite order, and $|\mathrm{Aut}_{k}(B)|<75b^{4}$. By (5) and Lemma 3.3, we obtain (i) and (ii).

(iii) By (6), Lemma 3.3 and Theorem 2.6, we have 
\begin{equation}
|G|<75 \cdot g \cdot (g-1) \cdot (8g+4) \cdot b^{4} \leq 3000 (g-1)^3b^4.
\end{equation}

(iv) Applying Theorem 4.1 to $f\colon S \rightarrow B$. By Remark 4.2 and Corollary 4.5, we have $\mathrm{Aut}_{B}(S) \subseteq \mathrm{Aut}_{B_{1}}(\widetilde{S}_{1})$. Let $\widetilde{g}$ be the arithmetic genus of the generic fiber of $\widetilde{f}_{1}$. By Lemma 4.6 and the inequality that $\widetilde{g} \leq g$, we have 
\begin{equation}
|\mathrm{Aut}_{B}(S)| \leq |\mathrm{Aut}_{B_{1}}(\widetilde{S}_{1})| \leq 8(\widetilde{g}-1)(2\widetilde{g}+1)^2 \leq 8(g-1)(2g+1)^2.
\end{equation}
Hence $|G|<8(g-1)(2g+1)^2 \cdot 75 b^{4}=600 (g-1)(2g+1)^2 \cdot b^4$.  
\end{proof}

\begin{cor}
Let $f\colon S \rightarrow B$ be a fibration over an algebraically closed field $k$ of characteristic $p > 0$, where $S$ is a minimal smooth projective surface of general type and $B$ is a smooth projective curve of genus $b\geq 2$. Let $C$ be the generic fiber of $f$, and $g:=p_{a}(C)$. Let $K$ be the function field of $B$, and $\widetilde{C}$ the normalization of $C \times_{K}\overline{K}$. We obtain the following assertions.
\\(i) If $g(\widetilde{C}) \geq 2$, then $|G| < 15658 \cdot (K_{S}^{2})^{4}$.
\\(ii) If $g(\widetilde{C})=1$, then $|G|< 1391 \cdot (K_{S}^{2})^{4}$.
\\(iii) If $g(\widetilde{C})=0$ and $p \geq 5$, then $|G|< 5500 \cdot (K_{S}^{2})^{4}$.
\\ As a conclusion, one has $|G|<15658 \cdot (K_{S}^{2})^{4}$ if $p \geq 5$.
\end{cor}

\begin{proof}  
Note that 
\begin{equation}
\begin{aligned}
\frac{gb}{2(g-1)(b-1)(1+\frac{1}{g})(1+\frac{g-1}{15g+1})}&\leq \frac{g}{(g-1)(1+\frac{1}{g})(1+\frac{g-1}{15g+1})}\\
&=\frac{15g^2+g}{16(g^2-1)}
\end{aligned}
\end{equation}
and the function $h(x):=\frac{15x^2+x}{x^2-1}$ is decreasing in the interval $[2,\infty)$, so we obtain
\begin{equation}
\frac{gb}{2(g-1)(b-1)(1+\frac{1}{g})(1+\frac{g-1}{15g+1})} \leq \frac{15 \cdot 4+2}{16 \cdot 3}=\frac{31}{24}.
\end{equation}
Combining this with Lemma 2.7, we have $gb\leq \frac{31}{24} K_S^2$. 

Similarly, we obtain
$$
\begin{aligned}
\frac{b(g-1)}{2(g-1)(b-1)(1+\frac{1}{g})(1+\frac{g-1}{15g+1})}\leq \frac{1}{(1+\frac{1}{g})(1+\frac{g-1}{15g+1})}=\frac{15g+1}{16(g+1)}<\frac{15}{16},
\end{aligned}
$$
$$
\begin{aligned}
\frac{b(2g+1)}{2(g-1)(b-1)(1+\frac{1}{g})(1+\frac{g-1}{15g+1})}&\leq \frac{2g+1}{(g-1)(1+\frac{1}{g})(1+\frac{g-1}{15g+1})}\\
&=\frac{(2g+1)(15g+1)}{16(g-1)(g+1)}\\
&\leq \frac{155}{48}.
\end{aligned}
$$
According to Lemma 2.7, we have $b(g-1)<\frac{15}{16} K_S^2$ and $b(2g+1)\leq \frac{155}{48} K_S^2$.
Therefore, by Theorem 4.7, we obtain the following assertions.

(i) If $g(\widetilde{C}) \geq 2$, then $|G| < 5625g^4b^4 \leq 5625\cdot (\frac{31}{24})^4(K_{S}^{2})^{4}<15658(K_{S}^{2})^{4}$.

(ii) If $g(\widetilde{C})=1$, then 
$$
|G|<1800 (g-1)b^4 \leq 1800 (g-1)^4b^4<1800\cdot (\frac{15}{16})^4(K_{S}^{2})^{4}<1391(K_{S}^{2})^{4}.
$$

(iii) If $g(\widetilde{C})=0$ and $p \geq 5$, then 
$$
|G|<600 (g-1)(2g+1)^2b^4 < 600 \left(\frac{15}{16}\right)^2 \left(\frac{155}{48}\right)^2(K_{S}^{2})^{4} <5500(K_{S}^{2})^{4}.
$$

\end{proof}

Next, we provide an example to explain that the exponent $4$ of the polynomial bound (in terms of $K_S^2$) is sharp.

\begin{example}
Let $k$ be an algebraically closed field of characteristic $p>0$, and let $B$ be the plane curve over $k$ defined by
$$
y^{p^n}z+yz^{p^n}=x^{p^n+1}\qquad (p^n\ge 3).
$$
Then $g(B)=\frac{1}{2}p^{n}(p^{n}-1)$ and $|\mathrm{Aut}_k(B)|=p^{3n}(p^{3n}+1)(p^{2n}-1)$ (see \cite{s13}). Let $S:=B \times B$. Then $S$ is a minimal smooth surface of general type. Let $p_{1} \colon B\times B \to B$ be the first projection.

Let $G:=\{ (\sigma_1,\sigma_2) \in \mathrm{Aut}_k(B) \times \mathrm{Aut}_k(B)$ $|$ $p_{1} \circ (\sigma_1,\sigma_2)=\sigma_{1}\circ p_{1} \}$.

\begin{displaymath}
\xymatrix{
S\ar[r]^{(\sigma_1,\sigma_2)} \ar[d]_{p_{1}} & S \ar[d]^{p_{1}}\\
B\ar[r]_{\sigma _{1}} & B 
}
\end{displaymath}
Then we have 
$$
\begin{aligned}
|G|
&=|\operatorname{Aut}_k(B)|^2
=p^{6n}(p^{3n}+1)^2(p^{2n}-1)^2
\sim p^{16n}\quad (n\to\infty),\\
K_S^2
&=8(g(B)-1)^2
=2(p^n-2)^2(p^n+1)^2
\sim 2p^{4n}\quad (n\to\infty).
\end{aligned}
$$
Therefore, in order to bound $|G|$, it is necessary to use at least the fourth power of $K_S^2$. In fact, ${1 \over 16} (K_{S}^{2})^{4}<|G|<35(K_{S}^{2})^4$.
\end{example}


\section*{Acknowledgements}
I would like to thank my advisor Prof. Wenfei Liu for suggesting this problem and for lots of discussions and encouragement. I would like to thank Prof. Qing Liu for very useful suggestions and generous help.


\end{document}